\documentclass[a4paper, 11pt, english]{article}

\usepackage{lscape}
\usepackage{pdflscape}

\usepackage{stmaryrd}

\usepackage{cmll}

\usepackage{amscd}
\usepackage{amssymb,amsmath,amsthm}
\usepackage{mathptmx}
\usepackage{mathrsfs}
\usepackage{color}
\usepackage{xspace}
\usepackage{bussproofs}
\EnableBpAbbreviations

\usepackage{tikz}

  \usepackage{txfonts}

\usepackage{bpextra}

\usepackage{enumerate}

\usepackage{centernot}

\usepackage{mathtools}

\usepackage{hyperref}

\usepackage{cleveref}

\usepackage{relsize}

\usepackage{caption}

\usepackage{multirow}

\usepackage{multicol}

\usepackage{array}
\newcolumntype{C}[1]{>{\centering\arraybackslash}p{#1}}
\newcolumntype{L}[1]{>{\arraybackslash}p{#1}}

\usepackage{rotating}

\newcommand{\mand}{\vartriangle}
\def\fCenter{\vdash}
\newcommand\mAND{\,\bigcap\,}
\def\mBAND{\mbox{\,\raisebox{0ex}{\rotatebox[origin=c]{180}{$\bigsqcup$}}\,}}
\def\mRA{\mbox{\,\raisebox{0ex}{\rotatebox[origin=c]{-90}{$\bigcap$}}\,}}
\def\mBRA{\mbox{\,\raisebox{0ex}{\rotatebox[origin=c]{90}{$\bigsqcup$}}\,}}
\def\aga{\texttt{a}}

\def\mbra{\mbox{$\,-{\mkern-3mu\blacktriangleright}\,$}}
\usepackage{mathdots}
\newcommand{\fns}{\footnotesize}
\usepackage{bussproofs}
\usepackage{qtree}
\usepackage{rotating}
\newcommand{\lc}{\langle}
\newcommand{\rc}{\rangle}
\EnableBpAbbreviations

\usepackage{amsmath}
\usepackage{amssymb}

\usepackage{enumerate}

\def\kn{\kern.1em}

\newtheorem{theo}{Theorem}[section]  

\newtheorem{prop}[theo]{Proposition}
\theoremstyle{definition}
\newtheorem{defi}[theo]{Definition}

\makeindex

\title{Multi-type Sequent Calculi}

\author{Sabine Frittella \and Giuseppe Greco \and Alexander Kurz \and
Alessandra Palmigiano\thanks{This research is supported by the NWO Vidi grant 016.138.314, the NWO Aspasia grant 015.008.054, and a Delft Technology Fellowship awarded to the fourth author in 2013.} \and Vlasta Sikimi\'{c}
}

\date{}

\begin{document}

\maketitle

\begin{abstract}
Display calculi are generalized sequent calculi which enjoy a `canonical' cut elimination strategy. That is, their cut elimination is uniformly obtained by verifying the assumptions of a meta-theorem, and is preserved by adding or removing structural rules. In the present paper, we discuss a proof-theoretic setting, inspired both to Belnap's Display Logic \cite{Belnap} and to Sambin's Basic Logic \cite{Sambin}, which generalises these calculi in two directions: by explicitly allowing different types, and by weakening the so-called \emph{display} and \emph{visibility} properties.

The generalisation to a multi-type environment makes it possible to introduce specific tools enhancing expressivity, which have proved useful e.g.\ for a smooth proof-theoretic treatment of multi-modal and dynamic logics \cite{Multitype, PDL}. The generalisation to a setting in which full display property is not required makes it possible to account for logics which admit connectives which are neither adjoints nor residuals, or logics that are not closed under uniform substitution. 

In the present paper, we give a general overview of the calculi which we refer to as \emph{multi-type calculi}, and we discuss their canonical cut elimination meta-theorem.

{\em Keywords:} cut elimination, display calculi, multi-type sequent calculi, non-classical logics, modal logic, dynamic logics.\\
{\em Math.\ Subject Class.\ 2010:} 03F05, 06D50, 06D05, 03G10, 06E15.
\end{abstract}

\section{Introduction}

The range of non-classical logics has been rapidly expanding, driven by influences from other fields which have opened up new opportunities for applications. The logical formalisms which have been developed as a result of this interaction have attracted the interest of a wider research community than the logicians, and their theory has been intensively investigated, especially w.r.t.\ their semantics and computational complexity.

However, most of these logics lack a comparable proof-theoretic development. More often than not, the hurdles preventing a standard proof-theoretic development for these logics are due precisely to some of their defining features which make them suitable for applications, such as e.g.\ their not being closed under uniform substitution, or the fact that (the semantic interpretations of) key connectives are not adjoints.

These difficulties caused the existing proposals in literature to be often ad hoc, not easily generalisable, and more in general lacking a smooth proof-theoretic behaviour. In particular, the difficulty in smoothly transferring results from one logic to another is a problem in itself, since these logics typically come in large families (consider for instance the family of dynamic logics), and hence proof-theoretic approaches which uniformly apply to each logic in a given family are in high demand (for an expanded discussion of the existing proof systems for dynamic epistemic logics, see \cite[section 3]{GAV}).

The problem of the transfer of results, tools and methodologies has been addressed in the proof-theoretic literature for the families of substructural and modal logics, and has given rise to the development of several generalisations of Gentzen sequent calculi (such as hyper-, higher level-, display- or labelled-sequent calculi).

The present paper focuses on the core technical aspects of a proof-theoretic methodology and set-up closely linked to Display logic \cite{Belnap} and Basic logic \cite{Sambin}. Instances of this set-up have appeared in \cite{GAV}, \cite{Multitype} and \cite{PDL} to account for (nonclassical versions of) Baltag-Moss-Solecki's dynamic epistemic logic and Propositional Dynamic Logic respectively. In ongoing work, this set-up is being applied to monotone modal logic, game logic, and linear logic. 

The present set-up generalizes display calculi in two directions: by allowing multi-type languages, and by dropping the so-called \emph{visibility} (i.e. all active formulas in all rules are in isolation or, equivalently, all the active contexts in all rules are empty), or \emph{segregation} property, a requirement which in principle substantially reduces the proof-power of the calculi which are to satisfy it, unless they also enjoy e.g.\ the display property. 

The additional expressivity provided by the multi-type setting allows for an improved proof-theoretic treatment of multi-modal and dynamic logics \cite{Multitype, PDL}. The generalisation to a setting in which display and visibility property are weakened makes it possible to account  both for logics---such as monotone modal logic---which admit connectives which are neither adjoints nor residuals, and for logics that are not closed under uniform substitution---such as dynamic epistemic logic.

The present paper is organized as follows: 
in section \ref{sec:multi}, we illustrate the main features of the multi-type setting without relying on one specific system in particular. In section \ref{sec:quasi}, we list a set of conditions generalising the (properly) displayable calculi of \cite{Wan02}. In Section \ref{sec:meta}, we discuss how these conditions guarantee the cut elimination meta-theorem for the multi-type calculi enjoying them.

\section{Multi-type calculi}
\label{sec:multi}

The present section is aimed at illustrating the environment of multi-type calculi. 
Our starting point is a propositional language, the terms of which form $n$ pairwise disjoint types $\mathsf{T_1}\ldots \mathsf{T_n}$.
As in display calculi, we allow both terms at the operational level---like for instance a propositional formula---and more terms at the structural level than the usual ones builded by means of Gentzen's comma---like for instance by using Belnap's star. We will use $a, b, c$ and $x, y, z$ to respectively denote operational and structural  terms of  unspecified (possibly different) type. Further, we assume that operational and structural connectives are given {\em both} within each type {\em and} also between different types.

As an example, consider the calculus given in \cite{Multitype} for the intuitionistic version of Baltag-Moss-Solecki's dynamic epistemic logic (IEAK) \cite{AMM}. This calculus admits four types: formulas $(\mathsf{Fm})$, actions $( \mathsf{Act})$, functional actions $(\mathsf{Fnc})$ and agents $(\mathsf{Ag})$. Typically, the languages of dynamic logics are constructed using e.g.\ agents or actions as parameters for the modal operators. The basic intuition is that the multi-type environment makes it possible to regard these parameters as first-class citizens, on  a par with formulas. Hence, IEAK-formulas such as $\langle\alpha\rangle \langle\aga\rangle A$---where $\langle\alpha\rangle$ and $\langle\aga\rangle$ are the modal operators associated with the (functional) action $\alpha$ and with the agent $\aga$, respectively---are translated in the multi-type setting as $\alpha \mand_0 (\aga \mand_2 A)$ using  binary connectives which take different types as arguments, as illustrated below:
\begin{center}
${\mand}_0 : \mathsf{Fnc} \times \mathsf{Fm} \to \mathsf{Fm} \qquad\qquad {\mand}_2 : \mathsf{Ag} \times \mathsf{Fm} \to \mathsf{Fm}$.
\end{center}
Each type has its own natural domain of interpretation.
For instance, the domain of interpretation for formulas can be a powerset algebra, or the algebra of downward-closed subsets of some poset (cf.\ \cite{KP}); the domain of interpretation of agents can be a set; finally, following \cite{Alexandru}, the domains of interpretation for actions and functional actions can be a quantale and a monoid, respectively (note that, unlike general actions, functional actions are typically not closed under non-deterministic choice).

\noindent Depending on the features of each domain, different properties can be required of the connectives. For instance, in the case of the connectives above, we stipulate that they be completely join-preserving in their second coordinate, and hence, by general order-theoretic facts, each of them has a right adjoint in its second coordinate. Therefore, the following additional connectives have a natural interpretation:
 \begin{center}
${\mbra}_0 : \mathsf{Fnc} \times \mathsf{Fm} \to \mathsf{Fm} \qquad\qquad {\mbra}_2 : \mathsf{Ag} \times \mathsf{Fm} \to \mathsf{Fm}$.
\end{center}


\noindent However, the present setting does not permit us to assume that the right adjoints in the first coordinate exist for either $\mand_0$ or $\mand_2$.
The lack of certain adjoints occurs in many situations, and is due to diverse reasons. For instance, by definition, the modal operators in monotone modal logic do not have adjoints. We will come back to this point later on.


\noindent Just like in standard display calculi, each operational connective corresponds to some structural connective (or proxy), that is interpreted in a context-sensitive way. To illustrate some introduction rules for operational connectives in the case of our running example, the translation mapping IEAK-formulas  into multi-type terms (of which the example above $\langle\alpha\rangle\langle\aga\rangle A\mapsto \alpha ?_0 (\aga\mand_2 A)$ is an instance) motivates  the following introduction rules for the binary connective $\mand_i$, for $i = 0, 2$:

{\footnotesize{
\begin{center}
\begin{tabular}{cc}
\AX$a \mAND_{\!\!i\,} A \fCenter X$
\UI$a \mand_i A \fCenter X$
\DisplayProof
 & 
\AX$x \fCenter a$
\AX$X \fCenter A$
\BI$x \mAND_{\!\!i\,} X \fCenter a \mand_i A$
\DisplayProof
\\
\end{tabular}
\end{center}
}}

\noindent as the natural counterparts of the following standard introduction rules for diamond-like modal operators:

{\footnotesize{
\begin{center}
\begin{tabular}{cc cc}
\AX$\{\aga\} A \fCenter X$
\UI$\lc\aga\rc A \fCenter X$
\DisplayProof
 &
 \AX$\{\alpha\} A \fCenter X$
\UI$\lc\alpha\rc A \fCenter X$
\DisplayProof
 &
\AX$X \fCenter A$
\UI$\{\aga\} X \fCenter \lc\aga\rc A$
\DisplayProof
&

\AX$X \fCenter A$
\UI$\{\alpha\} X \fCenter \lc\alpha\rc A$
\DisplayProof
\\
\end{tabular}
\end{center}
}}

\noindent 
Notice that, in the multi-type setting, each sequent is always interpreted in one domain. However, when  connectives take arguments of different types, premises of binary rules are of course interpreted in different domains.

\noindent Whenever  adjoints are available, display postulates such as the following ones can be added to the system. For $i = 0, 2$, the structural connectives $\mAND_{\!\!i\,}$ and $\mBRA_{i\,}$ are associated with $\mand_i$ and $\mbra_i$, respectively. Assuming that $\mand_i$ are left adjoints guarantees the soundness of the rules below whether or not the operational connectives $\mbra_i$ belong to the language.

{\footnotesize{
\begin{center}
\begin{tabular}{c}
\AX$x \mAND_{\!\!i\,} y \fCenter z$
\doubleLine
\UI$y \fCenter x {\mBRA_{\!i\,}} z$
\DisplayProof \\
\end{tabular}
\end{center}
}}

%
%
%

\noindent Most  types in the example above naturally arise  from the specification of the logic IEAK in its original presentation. However, the example below shows that additional types might be added for technical reasons. Consider for instance 
the following  rule {\em balance} introduced in \cite[section 5]{Multitype}:
\label{page:balance}

{\footnotesize{
\begin{center}
\AX$ X\fCenter Y$
\UI$ F \mAND_{\!\!0\,} X\fCenter F \mRA_{\!0\,} Y$
\DisplayProof
\end{center}
}}

\noindent This rule is sound for functional actions but not for general actions. Introducing functional and general actions as different types makes it possible for \emph{balance} to be closed under uniform substitution of arbitrary structures for congruent operational terms  \emph{within each type}, and thus for the whole system to satisfy conditions C$'_6$ and C$'_7$ below (see section \ref{sec:quasi}). 
 
 \noindent Our syntactic stipulations require that each term be assigned to a unique type {\em unambiguously}. This is a crucial requirement for the cut elimination meta-theorem of the next section, and will be explicitly stated in condition C$'_2$ below (see section \ref{sec:quasi}).
%
%
%
The notion of type-uniformity will be critical to the multi-type setting.
\begin{defi}
\label{type-uniformity}
A sequent $x\vdash y$ is {\em type-uniform} if $x$ and $y$ are of the same type.
\end{defi}
\noindent The proposition below is shown by a straightforward  induction on the minimal depth of derivation trees.
\begin{prop}
If a multi-type proof system is such that axioms are type-uniform and rules preserve type-uniformity, then every derivable sequent is type-uniform.
\end{prop}
In a display calculus, the cut rule is typically of the following form:

{\footnotesize{
\begin{center}
\AX$X \fCenter A$
\AX$A \fCenter Y$
\RightLabel{$Cut$}
\BI$X \fCenter Y$
\DisplayProof
\end{center}
}}

\noindent where $X, Y$ are structures and $A$ is a formula.
This translates straightforwardly to the multi-type environment, by the stipulation that  cut rules of the form

{\footnotesize{
\begin{center}
\AX$x \fCenter a$
\AX$a \fCenter y$
\RightLabel{$Cut$}
\BI$x \fCenter y$
\DisplayProof
\end{center}
}}

\noindent are allowed in the given multi-type system for each type.
Systems not enjoying the display property will be endowed with the following more general \emph{surgical cut}\label{surgical cut} rules:

{\footnotesize{
\begin{center}
\begin{tabular}{cc}
\AX$x \fCenter a$
\AX$(z \fCenter y) [a]^{pre}$
\RightLabel{$LCut$}
\BI$(z \fCenter y) [x/a]^{pre}$
\DisplayProof
&
\AX$(x \fCenter z) [a]^{suc}$
\AX$a \fCenter y$
\RightLabel{$R Cut$}
\BI$(x \fCenter z) [y/a]^{suc}$
\DisplayProof
\end{tabular}
\end{center}
}}

\noindent and the cut rules, like all other rules, will be required to preserve type-uniformity.

\noindent This finishes the discussion about the generalisation of display calculi along the multi-type dimension. The second dimension of our generalisation concerns weakening the \emph{visibility} property. The introduction rules for connectives will be still required to satisfy the visibility condition, but principal operational terms introduced by means of axioms will not be required anymore to occur in isolation (see condition C$'_5$, section \ref{sec:quasi}). However, a 
companion condition will be still required (see condition C$''_5$, section \ref{sec:quasi}), making sure that it is always possible to equivalently replace any  such axiom with a suitable one in which the given principal operational term occurs in isolation.

\section{Proper multi-type calculi}
\label{sec:quasi}

A multi-type calculus is  {\em proper} if it satisfies the following list of conditions (see \cite{GAV} for a discussion on C$'_5$ and C$''_5$):



\noindent \textbf{C$_1$: Preservation of operational terms.\;} Each operational term occurring in a premise of an inference rule {\em inf} is a subterm of some operational term in the conclusion of {\em inf}.

\noindent \textbf{C$_2$: Shape-alikeness of parameters.\;} Congruent parameters (i.e.\ non-active terms in the application of a rule) are occurrences of the same structure.

\noindent \textbf{C$'_2$: Type-alikeness of parameters.\;}  Congruent parameters have exactly the same type. This condition bans the possibility that a parameter changes type along its history.

\noindent \textbf{C$_3$: Non-proliferation of parameters.\;} Each parameter in an inference rule {\em inf} is congruent to at most one constituent in the conclusion of {\em inf}.

\noindent \textbf{C$_4$: Position-alikeness of parameters.\;} Congruent parameters are either all precedent or all succedent parts of their respective sequents. In the case of calculi enjoying the display property, precedent and succedent parts are defined in the usual way (see \cite{Belnap}). Otherwise, these notions can still be defined by induction on the shape of the structures, by relying on the polarity of each coordinate of the structural connectives.

\noindent \textbf{C$'_5$: Quasi-display of principal constituents.\;} If an operational term $a$ is principal in the conclusion sequent $s$ of a derivation $\pi$, then $a$ is in display, unless $\pi$ consists only of its conclusion sequent $s$ (i.e.\ $s$ is an axiom).

\noindent \textbf{C$''_5$: Display-invariance of axioms.} If $a$ is principal in an axiom $s$, then $a$ can be isolated by applying Display Postulates and the new sequent is still an axiom.

\noindent \textbf{C$'''_5$: Closure of axioms under surgical cut.} If $(x\vdash y)([a]^{pre}, [a]^{suc})$,  $a\vdash z [a]^{suc}$ and $v[a]^{pre}\vdash a$ are axioms, then $(x\vdash y)([a]^{pre}, [z/a]^{suc})$ and $(x\vdash y)([v/a]^{pre}, [a]^{suc})$  are again axioms. 

\noindent \textbf{C$'_6$: Closure under substitution for succedent parts within each type.\;} Each rule is closed under simultaneous substitution of arbitrary structures for congruent operational terms occurring in succedent position, {\em within each type}.

\noindent \textbf{C$'_7$: Closure under substitution for precedent parts within each type.\;} Each rule is closed under simultaneous substitution of arbitrary structures for congruent operational terms occurring in precedent position, {\em within each type}.

\noindent Condition C$_6'$ (and likewise C$_7'$) ensures, for instance, that if the following inference is an application of the rule $R$:

\begin{center}
\AX$(x \fCenter y) \big([a]^{suc}_{i} \,|\, i \in I\big)$
\RightLabel{$R$}
\UI$(x' \fCenter y') [a]^{suc}$
\DisplayProof
\end{center}

\noindent and $\big([a]^{suc}_{i} \,|\, i \in I\big)$ represents all and only  the occurrences of $a$ in the premiss which are congruent to the occurrence of $a$  in the conclusion (if $I = \varnothing$, then the occurrence of $a$ in the conclusion is congruent to itself), 
then also the following inference is an application of the same rule $R$:

\begin{center}
\AX$(x \fCenter y) \big([z/a]^{suc}_{i} \,|\, i \in I\big)$
\RightLabel{$R$}
\UI$(x' \fCenter y') [z/a]^{suc}$
\DisplayProof
\end{center}

\noindent where the structure $z$ is substituted for $a$.

\noindent This condition caters for the step in the cut elimination procedure in which the cut needs to be ``pushed up'' over rules in which the cut-formula in  succedent position  is parametric (cf.\ Section \ref{sec:meta}).

\noindent \textbf{C$'_8$: Eliminability of matching principal constituents.\;}
This condition requests a standard Gentzen-style checking, which is now limited to the case in which  both cut formulas are {\em principal}, i.e.~each of them has been introduced with the last rule application of each corresponding subdeduction. In this case, analogously to the proof Gentzen-style, condition C$'_8$ requires being able to transform the given deduction into a deduction with the same conclusion in which either the cut is eliminated altogether, or is transformed in one or more applications of the cut rule, involving proper subterms of the original operational cut-term. In addition to this, specific to the multi-type setting is the requirement that the new application(s) of the cut rule be also {\em strongly type-uniform} (cf.\ condition C$_{10}$ below).

\noindent \textbf{C$_9$: Type-uniformity of derivable sequents.} Each derivable sequent is type-uniform. 


\noindent \textbf{C$_{10}$: Preservation of type-uniformity of cut rules.} All cut rules preserve type-uniformity (cf.\  Definition \ref{type-uniformity}).

\section{Cut elimination meta-theorem}
\label{sec:meta}

\begin{theo}
\label{thm:meta multi}
Any multi-type sequent calculus satisfying C$_2$-C$_{10}$ is cut-admissible. If also C$_1$ is satisfied, then the calculus enjoys the subformula property.
\end{theo}

\begin{proof} 
We follow the proof in \cite[Section 3.3, Appendix A]{Wan02}. For the sake of conciseness, we will expand only on the parts of the proof which depart from the treatment in \cite{Wan02}. In particular, we consider elimination of surgical cuts (cf.\ page \pageref{surgical cut}). As to the principal move, the only difference concerns the case of a surgical cut application both premises of which are axioms. Condition C$'''_5$ guarantees that this cut application can be eliminated. The remaining principal moves go exactly as in \cite{Wan02}, thanks to C$'_8$. 
As to the parametric moves, we are in the following situation:

{\footnotesize{
\begin{center}
\AXC{\ \ \ $\vdots$ \raisebox{1mm}{$\pi_1$}}
\noLine
\def\fCenter{\vdash}
\UI$z \fCenter a$
\AXC{\ \ \ \ \ $\vdots$ \raisebox{1mm}{$\pi_{2.1}$}}
\noLine
\UI$(x_1 \fCenter y_1) [a_{u_1}]^{pre}$
\AXC{$\cdots{\phantom{\vdash}}$}
\AXC{\ \ \ \ \ $\vdots$ \raisebox{1mm}{$\pi_{2.n}$}}
\noLine
\UI$(x_n \fCenter y_n) [a_{u_n}]^{pre}$
\noLine
\TIC{\ \,$\rule[-5.2mm]{0mm}{0mm}\ddots\vdots\iddots\rule{0mm}{10mm}$ \raisebox{1mm}{$\pi_2$}}
\noLine
\UI$(x \fCenter y) [a]^{pre}$
\BI$(x \fCenter y) [z]^{pre}$
\DisplayProof
\end{center}
}}

\noindent where $(x \vdash y)[z]^{pre}[w]^{suc}$ means that $z$ and $w$ respectively occur in precedent and succedent position in $x \vdash y$, and the cut term $a$ is parametric in the conclusion of $\pi_2$ (the other case is symmetric). 

Conditions C$_2$-C$_4$ make it possible to follow the history of that occurrence of $a$, since these conditions enforce that the history takes the shape of a tree, of which we consider each leaf. Let $a_{u_i}$ (abbreviated to $a_u$ from now on) be one such uppermost-occurrence in the history-tree of the parametric cut term $a$ occurring in $\pi_2$, and let $\pi_{2.i}$ be the subderivation ending in the sequent $x_i \vdash y_i$, in which $a_u$ is introduced. 

Wansing's case (1) splits into two subcases: (1a) $a_u$ is introduced in display; (1b) $a_u$ is not introduced in display. Condition C$'_5$ guarantees that (1b) can only be the case when $a_u$ has been introduced via an axiom. 

If (1a), then we can perform the following transformation:

{\footnotesize{
\begin{center}
\begin{tabular}{lcr}
\bottomAlignProof
\AXC{\ \ \ \,$\vdots$ \raisebox{1mm}{$\pi_1$}}
\noLine
\UI$z \fCenter a$
\AXC{\ \ \ \ \,$\vdots$ \raisebox{1mm}{$\pi_{2.i}$}}
\noLine
\UI$a_u \fCenter y_i$
\noLine
\UIC{\ \ \ \,$\vdots$ \raisebox{1mm}{$\pi_2$}}
\noLine
\UI$(x \fCenter y) [a]^{pre}$
\BI$(x \fCenter y) [z]^{pre}$
\DisplayProof

 & $\rightsquigarrow$ &

\bottomAlignProof
\AXC{\ \ \ $\vdots$ \raisebox{1mm}{$\pi_1$}}
\noLine
\UI$z \fCenter a$
\AXC{\ \ \ \ \ $\vdots$ \raisebox{1mm}{$\pi_{2.i}$}}
\noLine
\UI$a_u \fCenter y_i$
\BI$z \fCenter y_i$
\noLine
\UIC{\ \ \ \ \ \ \ \ \ \ \ \ \ $\vdots$ \raisebox{1mm}{$\pi_2 [z/a]^{pre}$}}
\noLine
\UI$(x \fCenter y) [z]^{pre}$
\DisplayProof
 \\
\end{tabular}
\end{center}
}}

\noindent where $\pi_2 [z/a]^{pre}$ is the derivation obtained by substituting $z$ for every occurrence in the history of $a$. Notice that the assumption that $a$ is parametric in the conclusion of $\pi_2$ and that $a_u$ is principal in {\em inf} imply that $\pi_2$ has more than one node, and hence the transformation above results in a cut application of strictly lower height. Moreover, the assumptions that the original cut preserves type-uniformity (C$_{10}$), that every derivable sequent is type-uniform (C$_9$), and the type-alikeness of parameters (C$'_2$) imply that the sequent $a_u \vdash y_i$ is of the same type as the sequents $z \vdash a$. Hence, in particular, the new cut preserves type-uniformity. Finally, condition C$'_7$ implies that the substitution of $z$ for $a$ in $\pi_2$ gives rise to an admissible derivation $\pi_2[z/a]^{pre}$ in the calculus (use C$'_6$ for the symmetric case).
If (1b), i.e.\ if $a_u$ is the principal formula of an axiom, the situation is illustrated below in the derivation on the left-hand side:

{\footnotesize{
\begin{center}
\begin{tabular}{lcr}
\bottomAlignProof
\AXC{\ \ \ $\vdots$ \raisebox{1mm}{$\pi_1$}}
\noLine
\UI$z \fCenter a$
%
\AXC{$(x_i \fCenter y_i) [a_u]^{pre} [a]^{suc}$}
\noLine
\UIC{\ \ \ $\vdots$ \raisebox{1mm}{$\pi_2$}}
\noLine
\UI$(x \fCenter y) [a]^{pre}$
\BI$(x \fCenter y) [z]^{pre}$
\DisplayProof

 & $\rightsquigarrow$ &

\bottomAlignProof
\AXC{\ \ \,$\vdots$ \raisebox{1mm}{$\pi_1$}}
\noLine
\UI$z \fCenter a$
%
\AXC{$a_u \fCenter y' [a]^{suc}$}
\BI$z \fCenter y' [a]^{suc}$
\noLine
\UIC{\ \ \ $\vdots$ \raisebox{1mm}{$\pi'$}}
\noLine
\UI$(x_i \fCenter y_i) [z/a_u]^{pre} [a]^{suc}$
\noLine
\UIC{\ \ \ \ \ \ \ \ \ \ \ \ \ $\vdots$ \raisebox{1mm}{$\pi_2 [z/a]^{pre}$}}
\noLine
\UI$(x \fCenter y) [z]^{pre}$
\DisplayProof
 \\
\end{tabular}
\end{center}
}}

\noindent where $(x_i \fCenter y_i) [a_u]^{pre}[a]^{suc}$ is an axiom. Then, condition C$''_5$ implies that some sequent $a_u \fCenter y' [a]^{suc}$ exists, which is display-equivalent to the first axiom, and in which $a_u$ occurs in display. This new sequent can be either identical to $(x_i \fCenter y_i) [a_u]^{pre}[a]^{suc}$, in which case we proceed as in case (1a), or it can be different, in which case, condition C$''_5$ guarantees that it  is an axiom as well. Further, if $\pi$ is the derivation consisting of applications of display postulates which transform the latter axiom into the former, then let $\pi' = \pi[z/a]^{pre}$. As discussed when treating (1a), the assumptions imply that $\pi_2$ has more than one node, so the transformation described above results in a cut application of strictly lower height. Moreover, the assumptions that the original cut preserves type-uniformity, that every derivable sequent is type-uniform, and the type-alikeness of parameters imply that the sequent $a_u \vdash y' [a]^{suc}$ is of the same type as the sequent $z \vdash a$. Hence, the new cut is strongly type-uniform. Finally, condition C$'_7$ implies that substituting $z$ for $a$ in $\pi_2$ and in $\pi$ gives rise to admissible derivations $\pi_2 [z/a]$ and $\pi'$ in the calculus (use C$'_6$ for the symmetric case).

As to Wansing's case (2), assume that $a_u$ has been introduced as a parameter in the conclusion of $\pi_{2.i}$ by an application {\em inf} of the rule {\em Ru}. For instance, in the calculus of \cite[Section 5]{Multitype}, this situation can arise if $a$ is of type formula $\mathsf{Fm}$ and it is introduced by {\em weakening}, or if $a$ is of type functional actions $\mathsf{Fnc}$ and it is introduced by the rule {\em balance} or {\em atom}. Since $a_u$ is a leaf in the history-tree of $a$, this implies that $a_u$ is congruent only to itself in $\pi_{2.i}$. Hence, conditions C$'_7$, the assumption that the original cut is quasi strongly type-uniform, and the type-alikeness of parameters (C$'_2$) imply that the sequent $(x_i \fCenter y_i) [a_u]^{pre}$ can be replaced in the conclusion of $\pi_{2.i}$ by the sequent $(x_i \fCenter y_i) [z/a_u]^{pre}$ by means of an application of the same rule {\em Ru}. Let $\pi_{2.i}'$ be the resulting derivation. 

Therefore, the transformation below yields a derivation where $\pi_1$ is not used at all and the cut disappears.

{\footnotesize{
\begin{center}
\begin{tabular}{lcr}
\bottomAlignProof
\AXC{\ \ \ \,$\vdots$ \raisebox{1mm}{$\pi_1$}}
\noLine
\UI$z \fCenter a$
\AXC{\ \ \ \ \ \,$\vdots$ \raisebox{1mm}{$\pi_{2.i}$}}
\noLine
\UI$(x_i \fCenter y_i) [a_u]^{pre}$
\noLine
\UIC{\ \ \ \ $\vdots$ \raisebox{1mm}{$\pi_2$}}
\noLine
\UI$(x \fCenter y) [a]^{pre}$
\BI$(x \fCenter y) [z]^{pre}$
\DisplayProof

& $\rightsquigarrow$ &

\bottomAlignProof
\AXC{\ \ \ \ \ $\vdots$ \raisebox{1mm}{$\pi_{2.i}'$}}
\noLine
\UI$(x_i \fCenter y_i) [z/a_u]^{pre}$
\noLine
\UIC{\ \ \ \ \ \ \ \ \ \ \ \ \ $\vdots$ \raisebox{1mm}{$\pi_2 [z/a]^{pre}$}}
\noLine
\UI$(x \fCenter y) [z]^{pre}$
\DisplayProof
\end{tabular}
\end{center}
}}

From this point on, the proof proceeds like in \cite{Wan02}. It is useful to emphasise that the need to combine principal and parametric moves arises in multi-type settings such as \cite{Multitype} not only because of contraction or additive rules, but also due to the presence of structural rules such as 

{\footnotesize{
\begin{center}
\begin{tabular}{cc}
\AX$(x \mBAND Y)\,; (x \mBAND Z) \fCenter W $
\UI$x \mBAND (Y\,; Z) \fCenter W$
\DisplayProof

 & 
\AX$W \fCenter (x \mAND Y) > (x \mRA Z)$
\UI$W \fCenter x \mRA (Y > Z)$
\DisplayProof
\end{tabular}
\end{center}
}}

\end{proof}


\begin{thebibliography}{10}

\bibitem{Alexandru}
A.~Baltag, B.~Coecke, and M.~Sadrzadeh.
\newblock Epistemic actions as resources.
\newblock {\em Journal of Logic and Computation}, 17(3):555--585, 2007.

\bibitem{Belnap}
Nuel Belnap.
\newblock Display logic.
\newblock {\em Journal of Philosophical Logic}, 11:375--417, 1982.

\bibitem{PDL}
S.~Frittella, G.~Greco, A.~Kurz, and A.~Palmigiano.
\newblock Multi-type display calculus for propositional dynamic logic.
\newblock {\em Journal of Logic and Computation}, 2014 (forthcoming).

\bibitem{Multitype}
S.~Frittella, G.~Greco, A.~Kurz, A.~Palmigiano, and
  V.~Sikimi\'{c}.
\newblock A multi-type display calculus for dynamic epistemic logic.
\newblock {\em J.~of Logic and Comp.}, 2014 (forthcoming).

\bibitem{GAV}
S.~Frittella, G.~Greco, A.~Kurz, A.~Palmigiano, and
  V.~Sikimi\'{c}.
\newblock A proof-theoretic semantic analysis of dynamic epistemic logic.
\newblock {\em JLC}, 2014 (forthcoming).

\bibitem{Sambin}
C.~Faggian, G.~Sambin, and G.~Battilotti.
\newblock Basic logic: Reflection, symmetry, visibility.
\newblock {\em Journal of Symbolic Logic}, 65, 2000.

\bibitem{KP}
A.~Kurz and A.~Palmigiano.
\newblock Epistemic updates on algebras.
\newblock {\em Logical Methods in Computer Science}, 9(4), 2013.
\newblock DOI: 10.2168/LMCS-9(4:17)2013.

\bibitem{AMM}
M.~Ma, A.~Palmigiano, and M.~Sadrzadeh.
\newblock Algebraic semantics and model completeness for intuitionistic public
  announcement logic.
\newblock {\em Annals of Pure and Applied Logic}, 165:963--995, 2014.
\newblock http://dx.doi.org/10.1016/j.apal.2013.11.004.

\bibitem{Wan02}
H.~Wansing.
\newblock Sequent systems for modal logics.
\newblock {\em Hand.~of Phi.~Logic}, 8:61--145, 2002.

\end{thebibliography}
\end{document}